\documentclass{amsart}
\usepackage{graphicx}
\usepackage{amsfonts}
\usepackage{float}
\newtheorem{tm}{Theorem}

\newtheorem{defi}{Definition}

\newtheorem{rem}{Remark}
\newtheorem{rems}{Remarks}
\newtheorem{lm}{Lemma}
\newtheorem{ex}{Example}

\newtheorem{prop}{Proposition}
\newtheorem{nota}{Notation}

\begin{document}
\title{On Descartes' rule of signs for hyperbolic polynomials}
\author{Vladimir Petrov Kostov}
\address{Universit\'e C\^ote d’Azur, CNRS, LJAD, France}
\email{vladimir.kostov@unice.fr}

\begin{abstract}
  We consider univariate real polynomials with all
  roots real and  
  with two sign changes in the sequence of their coefficients which are all
  non-vanishing. One of the changes is between the linear and the constant
  term. By Descartes' rule of signs, such degree $d$ polynomials have $2$
  positive and $d-2$ negative roots. We consider the sequences of the moduli
  of their roots on the real
  positive half-axis. When the moduli are distinct, we give the
  exhaustive answer to the question at which positions can the moduli of
  the two positive roots be.

  {\bf Key words:} real polynomial in one variable; hyperbolic polynomial;
  sign pattern; Descartes' rule of signs\\

{\bf AMS classification:} 26C10
\end{abstract}
\maketitle

\section{Introduction}

In the present text we consider real polynomials in one variable. Such a monic
degree $d$ polynomial $Q$ is representable in the form $Q:=\sum _{j=0}^da_jx^j$,
$a_j\in \mathbb{R}$, $a_d=1$. The polynomial $Q$ is {\em hyperbolic}
if it has $d$ real roots
counted with multiplicity. We are interested in the {\em generic} case when
all coefficients $a_j$ are non-zero and the moduli of all roots are distinct.

Descartes' rule of signs (see \cite{Ca}, \cite{Cu}, \cite{DG}, \cite{Des},
\cite{Fo}, \cite{Ga}, \cite{J}, \cite{La} or \cite{Mes}) states that the
number $pos$ of
positive roots of $Q$
is majorized by the number $\tilde{c}$ of sign changes in its sequence of
coefficients. When applying this rule to $Q(-x)$, one deduces that the number
$neg$ of negative roots is majorized by the number $\tilde{p}$
of sign preservations.
As $pos+neg=\tilde{c}+\tilde{p}=d$, one finds that for hyperbolic polynomials,
$pos=\tilde{c}$ and $neg=\tilde{p}$. Moreover ${\rm sgn}(a_0)=(-1)^{pos}$.

The tropical version of Descartes' rule of signs is dealt with in
\cite{FoNoSh}. The problem treated in the present paper is part of
a more general problem about real univariate, but not necessarily
hyperbolic polynomials, considered in~\cite{FoKoSh} (see also the references
therein).

\begin{defi}
  {\rm The signs of the coefficients of the polynomial $Q$ define the
    {\em sign pattern} $\sigma (Q):=({\rm sgn}(a_d)$, ${\rm sgn}(a_{d-1})$,
    $\ldots$,
    ${\rm sgn}(a_0))$. For monic polynomials one has  ${\rm sgn}(a_d)=+$ and
    knowing the sign pattern is the same as knowing the corresponding
    {\em change-preservation pattern} which is a $d$-vector whose $j$th
    component is $c$ (resp. $p$) if $a_{d+1-j}a_{d-j}<0$ (resp. if
    $a_{d+1-j}a_{d-j}>0$). Example: for $d=5$, to the sign pattern $(+,-,-,+,-,+)$
    corresponds the change-preservation pattern $cpccc$.}
\end{defi}

Given a generic monic hyperbolic polynomial $Q$, one can consider the
moduli of its roots which are $d$ distinct numbers on the positive
half-axis and note the positions of the moduli of the negative roots. The
positive roots are denoted by $\alpha _1<\cdots <\alpha _{pos}$ and the
moduli of negative roots by $\gamma _1<\cdots <\gamma _{neg}$. The definition
of {\em order} (in the sense of order of moduli of roots)
and the corresponding notation
should be clear from the following example:

\begin{ex}
  {\rm Suppose that $d=7$, $pos=4$ and $neg=3$. Suppose that}

  $$\gamma _1<\alpha _1<\gamma _2<\gamma _3<\alpha _2<\alpha _3<\alpha _4~.$$
  {\rm Then we say that the moduli of the roots of the polynomial define the
    order $NPNNPPP$, i.~e. the letters $N$ and $P$ indicate the relative
    positions on the positive half-axis of the moduli of the negative and
    positive roots.}
  \end{ex}

\begin{defi}
  {\rm (1) Given two $d$-vectors -- a change-preservation pattern
    and an order --
    we say that they are {\em compatible} if the number of
letters $c$ (resp. $p$) of the former is equal to the number of letters
$P$ (resp. $N$) of the latter.
\vspace{1mm}

(2) A compatible couple (change-preservation pattern, order)
(we say {\em couple} for short) is
{\em realizable} if there exists a monic generic hyperbolic polynomial
whose sign of the coefficients and whose order of the moduli of roots
define the given couple. In this case we also say that the first of the
  components of the couple is realizable with the second one and vice versa.}
  \end{defi}

The present paper studies realizability of couples, a question to whose answer 
Descartes' rule of signs gives no indication. 

\begin{defi}\label{deficanon}
  {\rm For each change-preservation pattern (or sign pattern)
    one defines its corresponding
    {\em canonical} order as follows. One reads the change-preservation
    pattern from the right and
    to each letter $c$ (resp. $p$) one puts into correspondence the letter
    $P$ (resp. $N$). Example: for $d=5$, to the change-preservation pattern
    $ccpcp$ (or, equivalently, to the sign pattern $(+,-,+,+,-,-)$) corresponds
    the canonical order $NPNPP$.}
\end{defi}

Each change-preservation or sign pattern is realizable with its canonical order,
see \cite[Proposition~1]{KoSe}. When a change-preservation or sign 
pattern is realizable {\em only} with
its corresponding canonical order, it is also called {\em canonical}.

\begin{rems}\label{remscanon}
  {\rm (1) It is shown in \cite{KoRM} that canonical are exactly these
    change-preservation patterns which have no isolated sign change and
    no isolated sign preservation, i.~e. which contain no three
    consecutive components $PNP$ or $NPN$. Hence canonical are exactly these
    sign patterns which have no four consecutive signs $(+,+,-,-)$,
    $(-,-,+,+)$, $(+,-,-,+)$ or $(-,+,+,-)$.
    \vspace{1mm}

    (2) An order realizable with a single sign pattern is called {\em rigid}.
    It turns
    out that rigid are exactly the {\em trivial} orders (when all roots are of
    the same sign) and the orders in which moduli of positive and negative
    roots alternate, see~\cite{KoPMD22}.}
  \end{rems}

In what follows we denote by $\Sigma _{i_1,i_2,\ldots ,i_s}$, $s=\tilde{c}+1$,
$i_1+\cdots +i_s=d+1$,
sign patterns beginning with $i_1$ signs $+$ followed by $i_2$ signs $-$
followed by $i_3$ signs $+$ etc. Couples in which
the sign pattern has just one sign change (i.~e. $\tilde{c}=1$) have been
considered in \cite[Theorem~1 and Corollary~1]{KoPuMaDe}.
The result of \cite{KoPuMaDe}
reads:  

\begin{tm}\label{tmc1}
  The sign pattern $\Sigma _{m,n}$, $1\leq n\leq m$, is realizable with
  and only with orders such that $\alpha _1<\gamma _{2n-1}$. For
  $1\leq m\leq n$, it is realizable with and only with orders such that
  $\gamma _{d-2m}<\alpha _1$.
  \end{tm}

In the particular case $m=n$, Theorem~\ref{tmc1} imposes no restriction
on $\alpha _1$, so in this case all compatible couples are realizable.

\begin{rems}\label{remsimir}
  {\rm (1) There exist two commuting involutions acting on sign and
    change-preservation patterns and orders:}

  $$i_m~:~Q(x)\mapsto (-1)^dQ(-x)~~~\, {\rm and}~~~\, i_r~:~
  Q(x)\mapsto x^dQ(1/x)/Q(0)~.$$
  {\rm The involution $i_m$ exchanges the letters $c$ and $p$, the letters
    $P$ and $N$ and the numbers $\tilde{c}$ and $\tilde{p}$. The involution
    $i_r$ reads change-preservation patterns and orders from the right. It
    preserves the numbers $\tilde{c}$ and $\tilde{p}$.
    The factors $(-1)^d$ and $1/Q(0)$ are introduced to preserve the set
    of monic polynomials. A given couple $C$ is realizable or not together
    with the remaining
    one or three
    couples obtained under the $\mathbb{Z}_2\times \mathbb{Z}_2$-action
    defined by $i_m$ and $i_r$. One has always $i_m(C)\neq C$, but one could
    have $i_r(C)=C$ or $i_mi_r(C)=C$.
    \vspace{1mm}
    
    (2) Using the involution $i_m$ one can
    reduce the study of realizability of couples to the case
    $\tilde{c}\leq d/2$. The sign patterns  
    $\Sigma _{1,d}$ and $\Sigma _{d,1}$ are canonical. For $m>1$ and $n>1$,
    couples with sign patterns $\Sigma _{m,n}$ have just one sign change,
    and it is isolated. In this sense they are 
    the closest to couples with canonical sign pattern.}
  \end{rems}
    
In this paper we give the exhaustive answer to the question of realizability
of couples with sign patterns of the form $\Sigma _{m,n,1}$, i.~e. with two
sign changes one of which (for $n>1$) is isolated. This is the second simplest
case when the sign pattern is not canonical. Partial results on it 
are obtained in~\cite{KoPuMaDe}. One can observe that
$i_r(\Sigma _{m,n,1})=\Sigma _{1,n,m}$.

\begin{nota}
  {\rm We remind that for $\tilde{c}=2$, we denote by $\alpha _1<\alpha _2$ the
    positive and by $0<\gamma _1<\cdots <\gamma _{d-2}$ the moduli of the
    negative roots. For a monic hyperbolic polynomial $Q$ realizing a couple
    $C$ with sign pattern $\Sigma _{m,n,1}$, we denote by $\nu (Q)$ or $\nu (C)$
    the index $j$ such that $\gamma _j<\alpha _2<\gamma _{j+1}$.}
\end{nota}

\begin{tm}\label{tmresult}
  (A) Suppose that $m>n$. 
  \vspace{2mm}
  
  (1) If $n=1$, then the sign pattern $\Sigma _{m,n,1}$ is canonical.
  \vspace{1mm}
  
  (2) If $n\geq 4$, then a couple $C$ is realizable if and only if
  $\alpha _1<\gamma _1$ and $\nu (C)\leq 2n-2$.
  \vspace{1mm}
  
  (3) If $n=2$ or $3$, then a couple $C$ is realizable if and only if
  either $\alpha _1<\gamma _1$ and $\nu (C)\leq 2n-2$ or 
  $\gamma _1<\alpha _1<\alpha _2<\gamma _2<\cdots <\gamma _{d-2}$.
  \vspace{2mm}
  
  (B) Suppose that $m=n>1$. (If $m=n=1$, then there are no negative roots
  and one has $\alpha _1<\alpha _2$.) 
  \vspace{2mm}
  
  (4) If $m=n\geq 4$, then a couple is realizable if and only if
  $\alpha_1<\gamma _1$.
  \vspace{1mm}

  (5) If $m=n=2$ or $m=n=3$, then a couple is realizable if and only if
  either $\alpha_1<\gamma _1$ or
  $\gamma _1<\alpha _1<\alpha _2<\gamma _2<\cdots <\gamma _{d-2}$.
  \vspace{2mm}
  
  (C) Suppose that $m<n$.
  \vspace{2mm}

  (6) If $m=1$ and $n\geq 3$, the sign pattern $\Sigma _{1,n,1}$ is canonical.
  \vspace{1mm}

  (7) If $m\geq 2$ and $n\geq 5$, then a couple $C$ is realizable if and
  only if $\nu (C)\geq d-2m=n-m$.
  \vspace{1mm}

  (8) The necessary and sufficient conditions for realizability of the
  remaining sign patterns are as follows:

  The sign pattern $\Sigma _{1,2,1}$ is realizable with all three orders
  $NPP$, $PNP$ and $PPN$.

  The sign pattern $\Sigma _{2,3,1}$ is realizable with and only with
  the orders $NPPNN$, $PPNNN$, $PNPNN$,
  $PNNPN$ and $PNNNP$.

  The sign pattern $\Sigma _{2,4,1}$ is realizable with and only with
  the orders $PNNPNN$, $PNNNPN$ and $PNNNNP$.

  The sign pattern $\Sigma _{3,4,1}$ is realizable with and only with
  the orders $PPNNNNN$, $PNPNNNN$, $PNNPNNN$, $PNNNPNN$, $PNNNNPN$ and
  $PNNNNNP$.

  \end{tm}

The rest of the paper is occupied by the proof of Theorem~\ref{tmresult}.

\section{Proof of part (A) of Theorem~\protect\ref{tmresult}}

Part (1) of Theorem~\ref{tmresult}
needs no proof, see part (1) of Remarks~\ref{remscanon}.
For $m>n$, Theorems~3 and~4 of
\cite{KoPuMaDe} say that either $\nu (Q)\leq 2n-1$ or
$\gamma _1<\alpha _1<\alpha _2<\gamma _2$. The latter possibility exists only
for $n=2$ and $n=3$, and it is realizable. All cases when
$\nu (Q)\leq 2n-2$ and $\alpha _1<\gamma _1$ are also realizable.
We prove here that the case $\nu (Q)=2n-1$
is not realizable from which parts (2) and (3) of Theorem~\ref{tmresult}
follow.

\begin{tm}\label{tm0}
  For $m>n\geq 2$, there exists no monic hyperbolic polynomial $Q$ defining the
  sign pattern $\Sigma _{m,n,1}$ and with $\nu (Q)=2n-1$.
\end{tm}

\begin{proof}
  We remind that $m+n=d$ and $Q=\sum _{j=0}^{d}a_jx^j$, $a_{d}=1$.
  The set $A$ of hyperbolic
  polynomials defining the sign pattern $\Sigma _{m,n,1}$ (this is a subset
  of $Oa_0\ldots a_{d-1}$) is open and connected,
  see~\cite[Theorems~2 and~3]{KoAnn}. One knows that the subset of $A$ of
  polynomials with $\nu (.)\leq 2n-2$ is non-empty
  (\cite[Theorem~4]{KoPuMaDe}). If its subset of
  polynomials with $\nu (.)=2n-1$ is also non-empty, then by continuity
  there exists a polynomial $Q$ with $\alpha _2=\gamma _{2n-1}$
  (all other moduli of roots being distinct and $\alpha _1<\gamma _1$). One can
  perform a linear change $x\mapsto gx$, $g>0$, to obtain the condition
  $\alpha _2=\gamma _{2n-1}=1$. So now we concentrate on proving the following
  theorem from which Theorem~\ref{tm0} results:

\begin{tm}\label{tm1}
  For $m>n\geq 2$, there exists no monic hyperbolic polynomial defining the sign
  pattern $\Sigma _{m,n,1}$ and with roots $\alpha _1$, $\alpha _2$ and
  $-\gamma _i$, where

  $$0<\alpha _1<\gamma _1\leq \cdots \leq \gamma _{2n-2}\leq 1=\gamma _{2n-1}
  =\alpha _2<\gamma _{2n}\leq \cdots \leq \gamma _{d-2}~.$$
\end{tm}
\end{proof}

\begin{proof}[Proof of Theorem~\ref{tm1}]
We prove Theorem~\ref{tm1} first in the particular case $m=n+1$:
  
\begin{prop}\label{prop1}
  There exists no monic hyperbolic polynomial defining the sign pattern
  $\Sigma _{n+1,n,1}$ and with roots $\alpha _1$, $\alpha _2$ and $-\gamma _i$,
  where

  $$0<\alpha _1<\gamma _1\leq \cdots \leq \gamma _{2n-2}\leq 1=
  \gamma _{2n-1}=\alpha _2~.$$
\end{prop}

The proposition is proved after the proof of Theorem~\ref{tm1}. The proof of
Theorem~\ref{tm1} is performed by induction on $m$, where the induction
base is the case $m=n+1$.

\begin{lm}\label{lmQQ1}
  Suppose that $d\geq 3$ and $\sigma (Q)=\Sigma _{m,n,1}$. Set
  $Q=(x+\gamma _{i})Q_1$, $1\leq i\leq d-2$. Then $\sigma (Q_1)=\Sigma _{m_1,n_1,1}$, where either
  $m_1=m$, $n_1=n-1$ or $m_1=m-1$, $n_1=n$.
\end{lm}

\begin{proof}
  Indeed, it is clear that $Q_1(0)>0$. If the coefficient of $x$ in $Q_1$
is positive, then the one in $Q$ is also positive and
the polynomial $Q$ cannot define the sign pattern
$\Sigma _{m,n,1}$. Hence $Q_1$ defines a sign pattern of the form
$\Sigma _{m_1,n_1,1}$, $m_1+n_1=m+n-1=d-1$. If $n_1>n$ (resp. if $n_1<n-1$),
then the coefficient of
$x^{n+1}$
in $Q$ is negative (resp. the coefficient of $x^{n}$ is positive)
and the sign pattern of $Q$ is not $\Sigma _{m,n,1}$.
\end{proof}

Suppose that $m>n+1$. Define the polynomial $Q_1$ as in Lemma~\ref{lmQQ1}.
If $n_1=n-1$, then by \cite[Theorem~4]{KoPuMaDe},
$\nu (Q_1)\leq 2n-3$, so $\nu (Q)\leq 2n-3$. This means that
$\alpha _2<\gamma _{2n-2}$ -- a contradiction. 
If $n_1=n$, then by induction
hypothesis such a polynomial $Q_1$ does not exist, so $Q$ does not exist either.

\end{proof}

\begin{proof}[Proof of Proposition~\ref{prop1}]
  Suppose that there exists a polynomial of the form

$$Q:=\sum _{j=0}^{2n+1}t_jx^j=(x^2-1)R~,~~~\, R:=x^{2n-1}+b_{2n-2}x^{2n-2}+\cdots
+b_0$$
defining the sign pattern $\Sigma _{n+1,n,1}$, where the roots of the factor $R$
are $\alpha _1$, $-\gamma _1$, $\ldots$,
$-\gamma _{2n-2}$, $\gamma _i\in (0,1]$. Then

$$t_j=b_{j-2}-b_j~.$$
We represent the polynomial $R$ in the form

$$R=(x-\alpha _1)(x^{2n-2}+e_1x^{2n-3}+\cdots +e_{2n-2})~,$$
where $e_j$ is the $j$th elementary symmetric polynomial of the quantities
$\gamma _1$, $\ldots$, $\gamma _{2n-2}$. Hence

\begin{equation}\label{equt}
  \begin{array}{l}b_{n-1}=-\alpha _1e_{n-1}+e_n~~~\, {\rm and}~~~\,
  b_{n+1}=-\alpha _1e_{n-3}+e_{n-2}~,~~~\, {\rm so}\\ \\
  t_{n+1}=-\alpha _1e_{n-1}+e_n+\alpha _1e_{n-3}-e_{n-2}~.\end{array}
  \end{equation}
As $t_0>0$ and $t_1<0$, one obtains that $b_0<0$ and $b_1>0$, i.~e. the
factor $R$ defines the sign pattern $\Sigma _{2n-1,1}$. On the other hand

$$b_0=-\alpha _1\gamma _1\cdots \gamma _{2n-2}~~~\, {\rm and}~~~\,
b_1=-b_0(1/\alpha _1-1/\gamma _1-\cdots -1/\gamma _{2n-2})~,~~~\,
{\rm therefore}$$

\begin{equation}\label{equineq}
  1/\alpha _1-1/\gamma _1-\cdots -1/\gamma _{2n-2}>0~.
\end{equation}

\begin{lm}\label{lm1}
  (1) For $\gamma _j\in (0,1]$, it is impossible to have simultaneously
    the inequality $t_{n+1}>0$ (see 
    (\ref{equt})) and the equality

    \begin{equation}\label{equalpha}
1/\alpha _1-1/\gamma _1-\cdots -1/\gamma _{2n-2}=0~.
    \end{equation}

    (2) For $\gamma _j\in (0,1]$, it is impossible to have simultaneously
    the inequality $t_{n+1}>0$ and the inequality~(\ref{equineq}).
\end{lm}

Part (2) of the lemma finishes the proof of Proposition~\ref{prop1}.
\end{proof}

\begin{proof}[Proof of Lemma~\ref{lm1}]
  Part (1). Indeed, the condition $t_{n+1}>0$ is equivalent to

  \begin{equation}\label{equ>0}
    -e_{n-1}+e_{n-3}-(1/\alpha _1)(e_{n-2}-e_n)>0~.
    \end{equation}
  Substituting $G:=1/\gamma _1+\cdots +1/\gamma _{2n-2}$ for $1/\alpha _1$ in
  the last inequality one obtains

  \begin{equation}\label{eque}
\tau :=-e_{n-1}+e_{n-3}-G(e_{n-2}-e_n)>0~.
    \end{equation}
Set $I:=\{ 1; 2; \ldots ; 2n-2\}$. Denote by $Y$ a fixed product of
$n-3$ quantities $\gamma _i$ with distinct indices (hence a summand of
$e_{n-3}$) and by $J$ the
$(n-3)$-tuple of these indices. Further we write $Y(J)$.
Set $K:=I\setminus J$.
The coefficient of $Y(J)$ in $e_{n-3}$ equals~$1$.
There are exactly ${n+1\choose 2}$ products of $n-1$ quantities
$\gamma _j$ with distinct
indices whose $(n-1)$-tuple contains the $(n-3)$-tuple $J$.
Indeed, the remaining two indices must be among the $n+1$ indices of the
set~$K$.

On the other hand in the sum
$\sum _{J\subset I}(\sum _{k, \ell \in K, k<\ell}\gamma _k\gamma _{\ell})
Y(J)$ each summand
of $e_{n-1}$ is counted ${n-1\choose 2}$ times.
Therefore one can write

$$e_{n-3}=\sum _{J\subset I}Y(J)~~~\, {\rm and}~~~\,
e_{n-1}=\sum _{J\subset I}S_2(J)Y(J)/{n-1\choose 2}~,~~~\
S_2(J):=\sum_{ k, \ell \in K, k<\ell}\gamma _k\gamma _{\ell}
~.$$
One can similarly set

$$\begin{array}{ll}
  e_{n-2}=\sum _{J\subset I}S_1(J)Y(J)/(n-2)~,&
  S_1(J):=\sum _{k\in K}\gamma _k~~~\, \, \, 
{\rm and}\\ \\ 
e_n=\sum _{J\subset I}S_3(J)Y(J)/{n\choose 3}~,&
S_3(J):=\sum _{k, \ell , m\in K, k<\ell <m}
\gamma _k\gamma _{\ell}\gamma _m~.\end{array}$$
The inequality
$\gamma _i\gamma _j\gamma _k\leq (\gamma _i^2+\gamma _j^2)\gamma _k/2$
implies that each term in the sum $S_3(J)$ either equals a product of terms
of $S_1(J)$ and $S_2(J)$ or is majorized by the half-sum of such terms.
Counting the terms in these three sums yields 

\begin{equation}\label{equS}
S_3(J)/{n+1\choose 3}\leq (S_1(J)/(n+1))(S_2(J)/{n+1\choose 2})~,
\end{equation}
so for the coefficient $S$ of $Y(J)$ in the product $-G(e_{n-2}-e_n)$
(see~(\ref{eque})) one gets 

\begin{equation}\label{equSbis}
  \begin{array}{lll}
    S&:=&-S_1(J)/(n-2)+S_3(J)/{n\choose 3}\\ \\ &\leq&-S_1(J)(1/(n-2)-
  S_2(J){n+1\choose 3}/(n+1){n\choose 3}{n+1\choose 2})~.\end{array}
\end{equation}
The following inequalities hold true:

\begin{equation}\label{equSter}
  GS_1(J)>S_1(J)\sum _{j\in K}1/\gamma _j\geq
  (n+1)^2
  \end{equation}
(we apply the inequality between the mean arithmetic and the mean
harmonic here). One finds directly that

\begin{equation}\label{equchoose}
  {n+1\choose 3}/(n+1){n\choose 3}{n+1\choose 2}=1/(n-2){n+1\choose 2}~.
  \end{equation}
Consider the product $Y(J)$ in the different terms of the left hand-side
of (\ref{eque}). Its coefficient equals

$$\begin{array}{lll}
 \kappa&:=&-S_2(J)/{n-1\choose 2}+1-GS_1(J)/(n-2)+GS_3(J)/{n\choose 3}
\\ \\ &<&-S_2(J)/{n+1\choose 2}+1-GS_1(J)/(n-2)+
GS_1(J)S_2(J)/(n-2){n+1\choose 2}\\ \\ &=&(1-S_2(J)/{n+1\choose 2})
(1-GS_1(J)/(n-2))~.\end{array}$$
For the first term of the second line we use the inequality
${n+1\choose 2}>{n-1\choose 2}$; for the rightmost term of the second line
we use (\ref{equS}) and (\ref{equchoose}). The second factor is negative,
see (\ref{equSter}). The first factor is
non-negative, because $\gamma _i\in (0,1]$. Hence $\kappa <0$. This is the case
  of $\kappa$ defined for any product $Y(J)$, so $\tau <0$. This
  contradicts~(\ref{eque}).
  \vspace{1mm}

  Part (2). For fixed quantities $\gamma _i\in (0,1]$,
    the left hand-side of (\ref{equ>0})
    decreases as $1/\alpha _1$ increases. This follows from $e_{n-2}\geq e_n$.
    To prove the latter inequality we denote by $Z(J^*)$ the product of
    $(n-2)$ quantities $\gamma _j$ with distinct indices whose $(n-2)$-tuple
    is denoted by $J^*$. Then

    $$e_n=\sum _{J^*\subset I}\left( \left( \sum _{k,\ell \in I\setminus J^*,k<\ell}
      \gamma _k\gamma _{\ell}\right) /{n\choose 2} \right) Z(J^*)~.$$
      The numerator of the coefficient of $Z(J^*)$ contains ${n\choose 2}$
      summands which are $\leq 1$. Hence
      $e_n\leq \sum _{J^*\subset I}Z(J^*)=e_{n-2}$ and part (2) results from
      part (1).
\end{proof}

\section{Proof of part (B) of Theorem~\protect\ref{tmresult}}

One knows (see the lines after Theorem~\ref{tmc1}) that for $\tilde{c}=1$,
all sign patterns
$\Sigma _{n,n}$ are realizable with all possible orders. Hence there exist
degree $d-1$ polynomials $Q_{\flat}$ with roots $-\gamma _1$, $\ldots$,
$-\gamma _{d-2}$ and $\alpha _2$, where $\alpha _2<\gamma _1$ or
$\gamma _k<\alpha _2<\gamma _{k+1}$, $1\leq k\leq d-3$, or
$\gamma _{d-2}<\alpha _2$. One constructs then $Q$ in the form
$(x-\alpha _1)Q_{\flat}$, where $\alpha _1>0$ is very close to $0$.
Thus the respective coefficients of the polynomials $xQ_{\flat}$ and $Q$
(except their
constant terms) have the same signs and one has $\sigma (Q)=\Sigma _{n,n,1}$.
Therefore any order in which $\alpha _1<\gamma _1$ is realizable with
the sign pattern $\Sigma _{n,n,1}$.

On the other hand for $\tilde{c}=2$ and $n\geq 4$,
one has $\alpha _1<\gamma _1$, while 
for $n=2$ and $n=3$, if $\gamma _1<\alpha _1$, then
$\gamma _1<\alpha _1<\alpha _2<\gamma _2<\cdots <\gamma _{d-2}$
and this order is realizable with
the sign pattern $\Sigma _{n,n,1}$, see \cite[Theorem~3]{KoPuMaDe}.
This proves part~(B).

\section{Proof of part (C) of Theorem~\protect\ref{tmresult}}

Part (6) of the theorem follows from part (1) of Remarks~\ref{remscanon}.
\vspace{1mm}

To prove parts (7) and (8) of Theorem~\ref{tmresult} we use some lemmas whose
proofs are given in Section~\ref{secprlm}.

\begin{lm}\label{lmn-1n}
  For a polynomial $P$ defining the sign pattern $\Sigma _{n-1,n,1}$, $n\geq 5$, 
  one has $\nu (P)\geq 1$. All couples $C$ with sign pattern
  $\Sigma _{n-1,n,1}$, $n\geq 5$, and $\nu (C)\geq 1$ are realizable. 
  \end{lm}

\begin{rem}\label{rem341}
  {\rm Lemma~\ref{lmn-1n} does not hold true for $n=4$ as shown by the
    following example:}

$$\begin{array}{lll}
  P&:=&(x+1.01)^5(x-1)(x-0.1)\\ \\
  &=&x^7+3.95x^6+4.746x^5-0.41309x^4-5.11019095x^3\\ \\ &&
  -3.642011005x^2-0.63580905x+
  0.105101005~.\end{array}$$
  {\rm If one perturbs the five-fold root at $-1.01$ to obtain five distinct
    real roots close to it, one gets a
    polynomial $P$ with $\nu (P)=0$.}
\end{rem}

\begin{lm}\label{lm2n1}
  If for the hyperbolic polynomial $Q$ one has $\sigma (Q)=\Sigma _{2,n,1}$,
  $n\geq 4$, then $\nu (Q)\geq n-2$. 
  \end{lm}

\begin{lm}\label{lm351}
If for the hyperbolic polynomial $Q$ one has $\sigma (Q)=\Sigma _{3,5,1}$,
  then $\nu (Q)\geq 2$.
  \end{lm}

We prove part (7) of Theorem~\ref{tmresult} by induction on $n$ and $m$.
The proof of part (7) for $\Sigma _{2,5,1}$,
$\Sigma _{3,5,1}$ and $\Sigma _{4,5,1}$ follows from Lemmas~\ref{lm2n1},
\ref{lm351} and \ref{lmn-1n} respectively. We observe that part (7) is true for
couples with sign patterns $\Sigma _{1,n,1}$ (which are canonical).

Suppose that part (7) is proved

$$\begin{array}{ll}
  {\rm for}~~~n\leq n_0~~~ (n_0\geq 5),~~~1\leq m\leq n_0-1&{\rm and}\\ \\ 
{\rm for}~~~n=n_0+1,~~~1\leq m\leq m_0~~~ (1\leq m_0\leq n_0-1)~.&\end{array}$$
Consider a couple $C$ with sign
pattern $\Sigma _{m_0+1,n_0+1,1}$. If $m_0=n_0-1$,
then we apply Lemma~\ref{lmn-1n}.

Suppose that $1\leq m_0\leq n_0-2$. 
We apply Lemma~\ref{lmQQ1}. We represent a
polynomial $Q$ (with $\sigma (Q)=\Sigma _{m_0+1,n_0+1,1}$)
realizing the couple $C$ in the form
$(x+\gamma _{1})Q_1$. If $\sigma (Q_1)=\Sigma _{m_0,n_0+1,1}$, then by
inductive assumption

$$\nu (Q_1)\geq n_0+1-m_0>n_0+1-(m_0+1)~,~~~\, {\rm so}~~~\, 
\nu (Q)\geq n_0+1-(m_0+1)~.$$
If $\sigma (Q_1)=\Sigma _{m_0+1,n_0,1}$, then $\nu (Q_1)\geq n_0-(m_0+1)$ which
for $m_0\leq n_0-2$ is $>0$. In this case, as $\gamma _1$ is the smallest of
moduli of negative roots, the inequality $\nu (Q_1)\geq 1$ implies
$\gamma _2<\alpha _2$ and $\nu (Q)=\nu (Q_1)+1\geq n_0+1-(m_0+1)$.
This proves part~(7).
\vspace{1mm}

Part (8). The claim about $\Sigma _{1,2,1}$ follows from
\cite[Example~2]{KoPuMaDe}.

The statement about $\Sigma _{2,3,1}$ results from
\cite[Theorems~3 and~4]{KoPuMaDe} (for the orders $NPPNN$, $PNNNP$,
$PNNPN$, $PNPNN$ and the ones with which it is not realizable). It results
for the order $PPNNN$ from the following example:

$$\begin{array}{ll}&(x+1.3)(x+1.2)(x+1.1)(x-1)(x-0.5)\\ \\
  =&x^5+2.1x^4-0.59x^3-2.949x^2-0.419x+0.858~.\end{array}$$
The result concerning the sign pattern $\Sigma _{2,4,1}$ is proved in
\cite[Subsection~3.2]{KoSoz19}.

Realizability of the sign pattern $\Sigma _{3,4,1}$ with the order $PPNNNNN$
follows from Remark~\ref{rem341}. Its realizability with the orders
$PNPNNNN$, $PNNPNNN$, $PNNNPNN$, $PNNNNPN$ and $PNNNNNP$, as well as its
non-realizability with the remaining orders, can be deduced from
\cite[Theorems~3 and~4]{KoPuMaDe}.

\section{Proofs of Lemmas~\protect\ref{lmn-1n}, \protect\ref{lm2n1} and
  \protect\ref{lm351} \protect\label{secprlm}}

\begin{proof}[Proof of Lemma~\ref{lmn-1n}]

  The second claim of the lemma follows from \cite[Part~(2) of
    Theorem~4]{KoPuMaDe}, 
  so we prove only its first claim. One has to show that there exists
  no polynomial $P$ with $\sigma (P)=\Sigma _{n-1,n,1}$ and $\nu (P)=0$. Suppose
  that such a polynomial exists. Then there exists also a polynomial 
  $P$ with $\sigma (P)=\Sigma _{n-1,n,1}$
  and such that $\alpha _2=\gamma _1=1$ (this is proved by complete analogy
  with the beginning of the proof of Theorem~\ref{tm0}). Set

$$\begin{array}{ll}
  U:=\prod _{j=1}^{2n-3}(x+\gamma _j)=\sum _{j=0}^{2n-3}u_{2n-3-j}x^j~,~~~\, u_j>0~,&
  {\rm so}\\ \\ P=(x^2-(1+\alpha _1)x+\alpha _1)U=\sum _{j=0}^{2n-1}p_{2n-1-j}x^j~.&
\end{array}$$
Thus

$$p_{n-1}=u_{n-1}-(1+\alpha _1)u_{n-2}+\alpha _1u_{n-3}~.$$
We set $u_j:=e_j+e_{j-1}$, where $e_j$ is the $j$th elementary symmetric
polynomial of the quantities $\gamma _2$, $\ldots$, $\gamma _{2n-3}$, with
$e_0=1$ and $e_{-1}=0$. Hence

\begin{equation}\label{eque1}\begin{array}{l}
  e_{n-1}+e_{n-2}-(1+\alpha _1)(e_{n-2}+e_{n-3})+\alpha _1(e_{n-3}+e_{n-4})<0~,~~~\,
  {\rm i.~e.}\\ \\ e_{n-1}-\alpha _1e_{n-2}-e_{n-3}+\alpha _1e_{n-4}<0~.
\end{array}\end{equation}
We apply a reasoning similar to the one used in the proof of Lemma~\ref{lm1}.
Namely, we prove that it is impossible to have simultaneously
(\ref{eque1}) and

\begin{equation}\label{eqag1}
1/\alpha _1>\sum_{j=2}^{2n-3}1/\gamma _j~.
\end{equation}
To this end we first show that one cannot have at the same time
(\ref{eque1}) and

\begin{equation}\label{eqag2}
1/\alpha _1=\sum_{j=2}^{2n-3}1/\gamma _j~.
\end{equation}
So suppose that the couple of conditions (\ref{eque1}) and (\ref{eqag2})
is possible. Set $1/\alpha _1=S_{-1}:=\sum_{j=2}^{2n-3}1/\gamma _j$. This means that

\begin{equation}\label{equSS}
  S_{-1}(e_{n-1}-e_{n-3})-e_{n-2}+e_{n-4}<0~.
  \end{equation}
Set $I:=\{ 2; 3; \ldots ; 2n-3\}$. Denote by $Y$ a fixed product of
$n-4$ quantities $\gamma _i$ with distinct indices (hence a summand of
$e_{n-4}$) and by $J$ the
$(n-4)$-tuple of these indices. Further we write $Y(J)$.
Set $K:=I\setminus J$.
The coefficient of $Y(J)$ in $e_{n-4}$ equals~$1$.
There are exactly ${n\choose 2}$ products of $n-2$ quantities
$\gamma _j$ with distinct
indices whose $(n-2)$-tuple contains the $(n-4)$-tuple $J$.
Indeed, the remaining two indices must be among the $n$ indices of the
set~$K$.

On the other hand in the sum
$\sum _{J\subset I}(\sum _{k, \ell \in K, k<\ell}\gamma _k\gamma _{\ell})
Y(J)$ each summand
of $e_{n-2}$ is counted ${n-2\choose 2}$ times.
Therefore one can write

$$e_{n-4}=\sum _{J\subset I}Y(J)~~~\, {\rm and}~~~\,
e_{n-2}=\sum _{J\subset I}S_2(J)Y(J)/{n-2\choose 2}~,~~~\
S_2(J):=\sum_{ k, \ell \in K, k<\ell}\gamma _k\gamma _{\ell}
~.$$
One can similarly set

$$\begin{array}{ll}
  e_{n-3}=\sum _{J\subset I}S_1(J)Y(J)/(n-3)~,&
  S_1(J):=\sum _{k\in K}\gamma _k~~~\, \, \, 
{\rm and}\\ \\ 
e_{n-1}=\sum _{J\subset I}S_3(J)Y(J)/{n-1\choose 3}~,&
S_3(J):=\sum _{k, \ell , m\in K, k<\ell <m}
\gamma _k\gamma _{\ell}\gamma _m~.\end{array}$$
For a given $(n-4)$-tuple $J$, the coefficient of $Y(J)$ in the left
hand-side of (\ref{equSS}) equals

$$\phi :=S_{-1}(S_3(J)-S_1(J))-S_2(J)+1~.$$
The quantity $S_1(J)$ (resp. $S_3(J)$) contains $n$ (resp. ${n\choose 3}$)
terms. As all quantities $\gamma _j$, $j\geq 2$, are $>1$, one has

$$S_3(J)>S_1(J){n\choose 3}/n~.$$
Hence

$$\phi >\frac{{n\choose 3}-n}{{n\choose 3}}S_3(J)S_{-1}-S_2(J)+1~.$$
We set $S_{-1}(J)=\sum _{j\in K}1/\gamma _j$. One has

$$S_3(J)S_{-1}>S_3(J)S_{-1}(J)\geq {n\choose 3}nS_2(J)/{n\choose 2}~~~\,
{\rm and}$$

$$\frac{{n\choose 3}-n}{{n\choose 3}}\cdot
\frac{{n\choose 3}n}{{n\choose 2}}=
\frac{n}{n-1}\cdot \frac{(n-1)(n-2)-6}{3}=:\psi ~.$$
In this way

$$\phi >(\psi -1)S_2(J)+1~.$$
For $n\geq 5$, one has $\psi >1$ and $\phi >0$. Summing up over all
sets $J$ one obtains a contradiction with (\ref{equSS}).

If one supposes that (\ref{eque1}) and (\ref{eqag1})
take place simultaneously, then one
again arrives at a contradiction with (\ref{equSS}). Indeed, if in the product
$S_{-1}(e_{n-1}-e_{n-3})$ (see (\ref{equSS})) one replaces $S_{-1}$ by a larger
positive quantity, then the left hand-side of (\ref{equSS}) increases. This
is due to the fact that the symmetric elementary polynomials
$e_{n-1}$ and $e_{n-3}$ contain one and the same
number of summands, but the quantities $\gamma _j$ are $>1$, so
$e_{n-1}>e_{n-3}$. 
\end{proof}


  \begin{proof}[Proof of Lemma~\ref{lm2n1}]
    We use the involution $i_r$ (see part (1) of Remarks~\ref{remsimir})
    and consider polynomials
    defining the sign pattern $\Sigma _{1,n,2}$ instead of $\Sigma _{2,n,1}$.
    We denote again the positive roots by $\alpha _1<\alpha _2$ and the moduli
    of the negative roots by $\gamma _1<\cdots <\gamma _n$. One has
    $\gamma _n<\alpha _2$ (see \cite[Theorem~3]{KoPuMaDe} and part (1)
    of Remarks~\ref{remsimir}). In the new setting we have to prove that for no
    hyperbolic polynomial $Q$ with $\sigma (Q)=\Sigma _{1,n,2}$ does one have
    $\gamma _3<\alpha _1$. 

    Suppose that such a polynomial exists, so
    $\gamma _1$, $\gamma _2$, $\gamma _3\in (0,\alpha _1)$. We set

    $$\begin{array}{lll}
      A_1:=\alpha _1+\alpha _2~,&A_{-1}:=1/\alpha _1+1/\alpha _2~,&
      A_{-2}:=1/\alpha _1\alpha _2\\ \\ 
      G_1:=\gamma _1+\gamma _2+\gamma _3&
      G_{-1}:=1/\gamma _1+1/\gamma _2+1/\gamma _3~,&
      H_1:=\gamma _4+\cdots +\gamma _n~,\\ \\ 
      H_{-1}:=1/\gamma _4+\cdots +1/\gamma _n&
      L_{-2}:=\sum _{1\leq i<j\leq n}1/\gamma _i\gamma _j~,&
      \delta :=\alpha _1\alpha _2\gamma _1\cdots \gamma _n~.\end{array}$$
    The coefficients of $x^{n+1}$ and $x^2$ of $Q$ equal

    $$\begin{array}{llll}
      c_{n+1}&:=&-A_1+G_1+H_1&{\rm and}\\ \\
      c_2&:=&(-A_{-1}(G_{-1}+H_{-1})+A_{-2}+L_{-2})\delta ~.
    \end{array}$$
    We show that it is impossible to have simultaneously the inequalities

    \begin{equation}\label{equmany}
      \begin{array}{llc}
      \gamma _i\leq \alpha _1~,~~~\, i=1,~2,~3~,&&\gamma _j\geq \alpha _1~,~~~\,
      j\geq 4\\ \\
      c_{n+1}\leq 0&{\rm and}&c_2<0~.\end{array}
      \end{equation}
    Suppose that these inequalities except the last one hold true. We show
    that then the minimal possible value of $c_2/\delta$ is positive.
    Fix the sum
    $g_{12}:=\gamma _1+\gamma _2$. The terms in the expression for
    $c_2/\delta$ containing $\gamma _1$ or $\gamma _2$ are:

    $$((-A_{-1}+1/\gamma _3+H_{-1})g_{12}+1)/\gamma _1\gamma _2
    =:\eta /\gamma _1\gamma _2~.$$
    As $\gamma _3\leq \alpha _1$ and $H_{-1}\geq 1/\gamma _4$,
    the quantity $\eta$ is not smaller than

    $$(-1/\alpha _2+1/\gamma _4)g_{12}+1=
    (\gamma _1\alpha _2+\gamma _2\alpha _2+(\alpha _2-\gamma _1
    -\gamma _2)\gamma _4)/\alpha _2\gamma _4~.$$
    Conditions (\ref{equmany}) except the last of them imply
    $\alpha _2>\gamma _1+\gamma _2$. Hence $\eta >0$. This means that for fixed
    $g_{12}$, the quantity $c_2/\delta$ is minimal when $1/\gamma _1\gamma _2$ is
    minimal, i.~e. when $\gamma _1=\gamma _2$. In the same way one obtains
    that minimality of $c_2/\delta$ is possible only for
    $\gamma _1=\gamma _2=\gamma _3$ which we suppose to hold true from now on.

    Fix the sum $A_1$. The condition $c_{n+1}\leq 0$ implies

    $$A_1\geq 3\gamma _1+H_1\geq 3\gamma _1+\gamma _4~.$$
    The terms containing $\alpha _1$ or $\alpha _2$ in
    $c_2/\delta$ are:

    $$\begin{array}{ll}
      &(-A_1(3/\gamma _1+H_{-1})+1)A_{-2}
      \leq(-A_1(3/\gamma _1+1/\gamma _4)+1)A_{-2}\\ \\
      =&(-A_1(3\gamma _4+\gamma _1)+\gamma _1\gamma _4)A_{-2}/\gamma _1\gamma _4<
      (-(3\gamma _4+\gamma _1)(3\gamma _1+\gamma _4)+
      \gamma _1\gamma _4)A_{-2}/\gamma _1\gamma _4\\ \\
      =&(-3\gamma _4^2
      -9\gamma _1\gamma _4-3\gamma _1^2)A_{-2}/\gamma _1\gamma _4<0~.
      \end{array}$$
    This expression is minimal when $A_{-2}$ is maximal, i.~e. when
    $\alpha _1=\gamma _1$. In this case the coefficient which multiplies
    $1/\alpha _2$ in the quantity $c_2/\delta$ equals

    $$-3/\gamma _1-H_{-1}+1/\gamma _1=-2/\gamma _1-H_{-1}<0~,$$
    so $c_2/\delta$ is minimal when $1/\alpha _2$ is maximal, i.~e.
    when $\alpha _2$ is minimal, so $c_{n+1}=0$ and $\alpha _2=H_1+2\gamma _1$. 
    Set $H_{-2}:=\sum _{4\leq i<j\leq n}1/\gamma _i\gamma _j$.
    Thus we have

    $$L_{-2}=3/\gamma _1^2+(3/\gamma _1)H_{-1}+H_{-2}~~~\, {\rm and}$$

    $$\begin{array}{lll}
      c_2/\delta &=&-(1/\gamma _1+1/(H_1+2\gamma _1))(3/\gamma _1+H_{-1})\\ \\
      &&+
      1/\gamma _1(H_1+2\gamma _1)+3/\gamma _1^2+(3/\gamma _1)H_{-1}+H_{-2}\\ \\
      &=&-2/\gamma _1(H_1+2\gamma _1)+
      (2/\gamma _1)H_{-1}-H_{-1}/(H_1+2\gamma _1)+H_{-2}~.\end{array}$$
    For $n=4$, the term $H_{-2}$ is absent, one has $H_1=\gamma _4$,
    $H_{-1}=1/\gamma _4$, so

    $$c_2/\delta =-2/\gamma _1(\gamma _4+2\gamma _1)+
    2/\gamma _1\gamma _4-1/\gamma _4(\gamma _4+2\gamma _1)=
    3/\gamma _4(\gamma _4+2\gamma _1)>0~.$$
    For $n\geq 5$, one gets $H_{-1}>1/(H_1+2\gamma _1)$ and
    $H_{-2}(H_1+2\gamma _1)>H_{-1}$, so again $c_2/\delta >0$. Thus one cannot
    have all conditions (\ref{equmany}) fulfilled at the same time which proves
    the lemma.
\end{proof}


\begin{proof}[Proof of Lemma~\ref{lm351}]
  As in the proof of Lemma~\ref{lm2n1} we use the involution $i_r$ (see
  part (1) of Remarks~\ref{remsimir}).
  We consider polynomials
    defining the sign pattern $\Sigma _{1,5,3}$ instead of $\Sigma _{3,5,1}$.
    Thus for the positive roots $\alpha _1<\alpha _2$ and for the moduli
    of the negative roots $\gamma _1<\cdots <\gamma _6$ one has
    $\gamma _6<\alpha _2$ (see \cite[Theorem~3]{KoPuMaDe} and part (1)
    of Remarks~\ref{remsimir}). We have to show that for no
    hyperbolic polynomial $Q$ with $\sigma (Q)=\Sigma _{1,5,3}$ does one have
    $\gamma _5<\alpha _1$.

    Suppose that such a polynomial exists, so
    $\gamma _1$, $\ldots$, $\gamma _5\in (0,\alpha _1)$. Making a linear change
    $x\mapsto bx$, $b>0$, one obtains the condition $\alpha _1=1$. We set

    $$\begin{array}{lll}
      A_1:=\alpha _1+\alpha _2~,&A_{-1}:=1/\alpha _1+1/\alpha _2~,&
      A_{-2}:=1/\alpha _1\alpha _2\\ \\ 
      H_1:=\sum _{j=1}^6\gamma _j&
      H_{-1}:=\sum _{j=1}^61/\gamma _j~,&
      H_{-2}:=\sum _{1\leq i<j\leq 6}1/\gamma _i\gamma _j~,\\ \\ 
      H_{-3}:=\sum _{1\leq i<j<k\leq 6}1/\gamma _i\gamma _j\gamma _k~,&{\rm and}&
      \delta :=\alpha _1\alpha _2\gamma _1\cdots \gamma _6~.\end{array}$$
    The coefficients of $x^{7}$ and $x^3$ of $Q$ equal

     $$\begin{array}{llll}
      c_{7}&:=&-A_1+H_1&{\rm and}\\ \\
      c_3&:=&(-A_{-1}H_{-2}+A_{-2}H_{-1}+H_{-3})\delta ~.
    \end{array}$$
    We show that it is impossible to obtain simultaneously the conditions

    $$\gamma _1\leq \cdots \leq \gamma _5\leq 1\leq \gamma _6~,~~~\,
    c_7<0~~~\, {\rm and}~~~\, c_3<0~.$$

    We prove that for fixed sum $g_{1}:=\gamma _1+\gamma _2$,
    the quantity $c_3/\delta$ is minimal for $\gamma _1=\gamma _2$.
    To this end we set

    $$g_{-1}:=1/\gamma _1+1/\gamma _2=g_1/\gamma _1\gamma _2~,~~~\,
    G_{-1}:=1/\gamma _3+\cdots +1/\gamma _6~,~~~\, G_{-2}:=\sum _{3\leq i<j\leq 6}
    1/\gamma _i\gamma _j$$
    and observe that the terms in  $c_3/\delta$ containing
    $\gamma _1$ or $\gamma _2$ are:

    $$\begin{array}{ll}&-A_{-1}(1/\gamma _1\gamma _2+g_{-1}G_{-1})+
    A_{-2}g_{-1}+(1/\gamma _1\gamma _2)G_{-1}+
    g_{-1}G_{-2}\\ \\ =&(-A_{-1}(1+g_1G_{-1})+
    A_{-2}g_1+G_{-1}+g_1G_{-2})/\gamma _1\gamma _2\\ \\
    =&((-A_{-1}+G_{-1})+g_1(-A_{-1}G_{-1}+A_{-2}+G_{-2}))/\gamma _1\gamma _2~.
    \end{array}$$
    Obviously $-A_{-1}+G_{-1}>1/\gamma _3+1/\gamma _4\geq 2/\gamma _4$.
    Notice that $g_1\leq 2$.
    We show that

    \begin{equation}\label{equ*}
      1/\gamma _4-A_{-1}G_{-1}+A_{-2}+G_{-2}\geq 0
    \end{equation}
    which
    means that $c_3/\delta$ is minimal when $\gamma _1\gamma _2$ is maximal,
    i.~e. when $\gamma _1=\gamma _2$. Inequality (\ref{equ*})
    results from the two
    inequalities

    $$\begin{array}{llll}
      (1/\alpha _2)G_{-1}&\leq&
      1/\gamma _3\gamma _6+1/\gamma _4\gamma _6+1/\gamma _5\gamma _6+A_{-2}&
      {\rm and}\\ \\
      G_{-1}&\leq&1/\gamma _4+1/\gamma _3\gamma _4+1/\gamma _3\gamma _5+
      1/\gamma _4\gamma _5~.\end{array}$$
    The first of them follows from
    $1/\gamma _i\gamma _6\geq 1/\gamma _i\alpha _2$, $i=3$, $4$, $5$ and
    $1/\alpha _2\gamma _6\leq 1/\alpha _2$. The second results from

    $$G_{-1}\leq 1/\gamma _3+1/\gamma _4+2/\gamma _5\leq 1/\gamma _4+
    1/\gamma _3\gamma _4+1/\gamma _3\gamma _5+
    1/\gamma _4\gamma _5~.$$
    It should be noticed that if
    instead of $\gamma _1$ and $\gamma _2$ one chooses other two quantities
    $\gamma _i$ and $\gamma _j$, $1\leq i<j\leq 5$, the proof that
    $c_3/\delta$ is minimal for $\gamma _i=\gamma _j$ can be performed in the
    same way.

    Thus one needs to consider only the situation $\gamma _1=\cdots =\gamma _5$
    in which case

    $$c_3/\delta =-(1+1/\alpha _2)(10/\gamma _1^2+5/\gamma _1\gamma _6)+
    (1/\alpha _2)(5/\gamma _1+1/\gamma _6)+10/\gamma _1^3+
    10/\gamma _1^2\gamma _6~.$$
    The coefficient of $1/\alpha _2$ is

    $$(-10/\gamma _1^2+5/\gamma _1)+(-5/\gamma _1\gamma _6+1/\gamma _6)<0~.$$
    Hence $c_3/\delta$ is minimal when $\alpha _2$ is minimal, i.~e. when
    $\alpha _2=5\gamma _1+\gamma _6-1$. Further we shorten the notation as
    follows: we set $a:=\alpha _2$, $r:=\gamma _1$ and $w:=\gamma _6$.
    Hence $0<r\leq 1\leq w\leq a$. For $a=5r+w-1$, we compute the product

    $$ar^3wc_3/\delta =(5r+w-1)r^3wc_3/\delta =:-2K~.$$
    One obtains

    $$\begin{array}{rcl}
      K&=&12r^3+25r^2w+5rw^2-25r^2-30rw-5w^2+5r+5w\\ \\
      &=&-(1-r)(12r^2+25rw+5w^2-5r-5w)-8r^2<0~.\end{array}$$
    The minimal possible value of $c_3/\delta$ being positive, one cannot
    have $c_7<0$ and $c_3<0$ at the same time. This proves the lemma.
    \end{proof}

\end{document}